\documentclass[12pt]{amsart}
\usepackage{amsmath,amsthm,amsfonts,amssymb,times}

\usepackage{titlesec}
\titleformat{\section}  
{\normalfont\Large\bfseries}
{\thesection}{1em}{}
\titleformat{\subsection}
{\normalfont\large\bfseries}
{\thesubsection}{1em}{}

\numberwithin{equation}{section}  

\DeclareMathAlphabet{\curly}{U}{rsfs}{m}{n}  

\usepackage{enumerate}
\makeatletter
\@namedef{subjclassname@2010}{
  \textup{2010} Mathematics Subject Classification}

\makeatother
\newtheorem{thm}{Theorem}[section]

\newtheorem{lem}[thm]{Lemma}
\newtheorem{pro}[thm]{Proposition}
\theoremstyle{definition}

\numberwithin{equation}{section}
\numberwithin{equation}{section}

\DeclareMathOperator{\li}{li}

\makeatletter
\renewcommand{\pmod}[1]{\allowbreak\mkern7mu({\operator@font mod}\,\,#1)}
\makeatother

\newcommand{\bal}{\[\begin{aligned}}
\newcommand{\eal}{\end{aligned}\]}

\newcommand{\be}{\begin{equation}}
\newcommand{\ee}{\end{equation}}

\newcommand{\ssum}[1]{\sum_{\substack{#1}}}  

\renewcommand{\b}{\ensuremath{\beta}}
\newcommand{\del}{\ensuremath{\delta}}
\newcommand{\eps}{\ensuremath{\varepsilon}}
\newcommand{\g}{\gamma}

\renewcommand{\le}{\leqslant}
\renewcommand{\leq}{\leqslant}
\renewcommand{\ge}{\geqslant}
\renewcommand{\geq}{\geqslant}

\renewcommand{\(}{\left(}
\renewcommand{\)}{\right)}

\newcommand{\pfrac}[2]{\left(\frac{#1}{#2}\right)}  

\newcommand{\m}{\textup{meas}}
\newcommand{\re}{\textup{Re}}
\newcommand{\im}{\textup{Im}}

\begin{document}

\title[The prime number race and zeros of  $L$-functions off the critical line]{The prime number race and zeros of Dirichlet $L$-functions off the critical line. III}

\author{Kevin Ford}
\address{Department of Mathematics, University of Illinois at Urbana-Champaign,
1409 W. Green Street,
Urbana, IL, 61801}
\email{ford@math.uiuc.edu}

\author{Sergei Konyagin}
\address{Steklov Mathematical Institute, 8, Gubkin Street, Moscow, 119991,
Russia}
\email{konyagin@mi.ras.ru}

\author{Youness Lamzouri}
\address{Department of Mathematics, University of Illinois at Urbana-Champaign,
1409 W. Green Street,
Urbana, IL, 61801}
\email{lamzouri@math.uiuc.edu}

\begin{abstract}
We show, for any $q\ge 3$ and distinct reduced residues $a,b \pmod q$, the existence of certain hypothetical sets of
zeros of Dirichlet $L$-functions lying off the critical line implies that $\pi(x;q,a)<\pi(x;q,b)$ for
a set of real $x$ of asymptotic density 1.
\end{abstract}



\thanks{}

\maketitle

\section{Introduction}
For $(a,q)=1$, let $\pi(x;q,a)$ denote the number of primes $p\le x$ with $p\equiv
a\pmod{q}$.  The study of the relative magnitudes of the functions
$\pi(x;q,a)$ for a fixed $q$ and varying $a$ is known colloquially as
the ``prime race problem'' or ``Shanks-R\'enyi prime race problem''.
For a survey of problems and results on prime races, the reader may consult the papers
\cite{FK3} and \cite{GM}.  One basic problem is the study of
$P_{q;a_1\dots,a_r}$, the set of real numbers $x\geq 2$ such that $\pi(x;q,a_1)>\cdots>\pi(x;q,a_r)$.
It is generally believed that all sets $P_{q;a_1\dots,a_r}$ are unbounded.
Assuming the Generalized Riemann Hypothesis for Dirichlet $L$-functions modulo $q$ (GRH$_q$)
and that the nonnegative imaginary parts of zeros of these $L$-functions are linearly  independent over the rationals,
 Rubinstein and Sarnak \cite{RS} have shown for any $r$-tuple of reduced residue
classes $a_1,\ldots,a_r$ modulo $q$, that $P_{q;a_1,\ldots,q_r}$ has a positive logarithmic density (although the
density may be quite small in some cases).

 In \cite{FK} and \cite{FK2}, Ford and Konyagin investigated how possible violations of the Generalized
 Riemann Hypothesis (GRH) would affect prime number races.  In \cite{FK}, they proved that the
existence of certain sets of
zeros off the critical line would imply that some of the sets $P_{q;a_1,a_2,a_3}$ are bounded, giving a negative
 answer to the prime race problem with $r=3$.  Paper \cite{FK2} was devoted to similar questions for $r$-way prime
races with $r>3$.  One result from \cite{FK2} states that for any $q$, $r\le \phi(q)$ and set $\{a_1,\ldots,a_r\}$ of
reduced residues modulo $q$, the existence of certain hypothetical sets of zeros of Dirichlet $L$-functions modulo
$q$ implies that at most $r(r-1)$ of the sets $P_{q;\sigma(a_1),\ldots,\sigma(a_r)}$ are unbounded, $\sigma$ running over
all permutations of $\{a_1,\ldots,a_r\}$.

 In this paper, we  investigate the effect of zeros of $L$-functions lying off the critical line
for two way prime races.  This case is harder, since it is
 unconditionally proved that for certain races $\{q;a,b\}$ the set $P_{q;a, b}$ is unbounded.
For example, Littlewood \cite{Li} proved that $P_{4;3,1}$, $P_{4;1,3}$, $P_{3;1,2}$ and $P_{3;2,1}$ are unbounded.
Later Knapowski and Tur\`an (\cite{KT1}, \cite{KT2})  proved for many $q,a,b$ that $\pi(x;q,b)-\pi(x;q,a)$
changes sign infinitely often and more recently Sneed \cite{Sn} showed that $P_{q;a,b}$ is unbounded
for every $q\le 100$ and all possible pairs $(a,b)$.

 Nevertheless, we  prove that the existence of certain zeros off the critical line would imply that
the set $P_{q;a,b}$ has asymptotic density zero, in contrast with a conditional result of  Kaczorowski
\cite{Ka1} on GRH, which asserts that $P_{q;1,b}$ and $P_{q;b,1}$ have positive lower densities for all $(b,q)=1$.

Let $q\geq 3$ be a positive integer and $a,b$ be distinct reduced residues modulo $q$. Moreover, for any set $\mathcal{S}$ of real numbers we define
$\mathcal{S}(X)= \mathcal{S}\cap [2,X].$

%
%

\begin{thm}\label{MainThm}
Let $q\ge 3$ and suppose that $a$ and $b$ are distinct reduced residues modulo $q$.
Let $\chi$ be a nonprincipal Dirichlet character with $\chi(a)\ne \chi(b)$, and put $\xi=\arg(\chi(a)-\chi(b))
\in [0,2\pi)$.
Suppose  $\frac12<\sigma<1$, $0<\del<\sigma-\frac12$,$A>0$, and $\mathcal{B}=\mathcal{B}(\xi, \sigma, \del, A)$ is a multiset
of complex numbers satisfying the conditions listed in Section \ref{sec:construction}.
If $L(\rho,\chi)=0$ for all $\rho\in \mathcal{B}$,
$L(s,\chi)$ has no other zeros in the region $\{ s : \re(s)\ge \sigma-\del, \im(s)\ge 0\}$, and for all other
nonprincipal characters $\chi'$ modulo $q$, $L(s,\chi')\ne 0$ in the region
$\{s: \re(s)\ge \sigma-\del, \im(s)\ge 0\}$, then $$\lim_{X\to\infty}\frac{\m(P_{q;a,b}(X))}{X}=0.$$
\end{thm}

{\bf Remarks.}  Such $\chi$ exists whenever $a$ and $b$ are distinct modulo $q$.
The sets $\mathcal{B}$  have the property that any
$\rho\in \mathcal{B}$ has real part in $[\sigma-\delta,\sigma]$, imaginary part greater than $A$,
and multiplicity $O((\log \im(\rho))^{3/4})$ (that is, the multiplicities are much smaller than known bounds on the multiplicity of zeros of Dirichlet
$L$-functions).  The number of elements of $\mathcal{B}$ (counted with multiplicity) with
imaginary part less than $T$ is $O((\log T)^{5/4})$, and thus $\mathcal{B}$ is quite a ``thin'' set.
Also, we note that if $L(\b+i\g,\chi)=0$ then $L(\b-i\g,\overline{\chi})=0$, which is a consequence of the functional
equation for Dirichlet $L$-functions (See e.g. Ch. 9 of \cite{Da}).  The point of Theorem \ref{MainThm}
is that proving  $$\limsup_{X\to\infty}\frac{\m(P_{q;a,b}(X))}{X}>0$$ requires showing that the multiset
of zeros of $L(s,\chi)$ cannot contain any of the multisets $\mathcal{B}$.  This is beyond what is possible
with existing technology (see e.g. \cite{Iv} for the best known estimates for multiplicities of zeros).

Our method works as well for the difference $\pi(x)-\li(x)$, the error term in the prime number theorem.
Littlewood \cite{Li} established that this quantity changes sign infinitely often. Let $P_1$ be the set of
real numbers $x\geq 2$ such that $\pi(x)>\li(x)$.  In \cite{Ka2} Kaczorowski proved, assuming the Riemann Hypothesis,
that both $P_1$ and $\overline{P}_1$ have positive lower densities. Assuming the Riemann Hypothesis and that
the nonnegative imaginary parts of the zeros of the Riemann zeta function $\zeta(s)$ are linearly independent over the rationals,
 Rubinstein and Sarnak \cite{RS} have shown that $P_1$ has a positive logarithmic density
$\delta_1\approx 0.00000026$.  In contrast with these results we prove that the existence of certain zeros
 of $\zeta(s)$ off the critical line would imply that the set $P_1$ has asymptotic density zero
(or asymptotic density 1).

\begin{thm}\label{Thm2}
Suppose $\frac12<\sigma<1$, $0<\del<\sigma-\frac12$ and $A>0$.
(i) If $\xi=0$, $\mathcal{B}=\mathcal{B}(\xi, \sigma, \del, A)$ satisfies the conditions of Section \ref{sec:construction},
$\zeta(\rho)=0$ for all $\rho\in \mathcal{B}$,
and $\zeta(s)$ has no other zeros in the region $\{ s : \re(s)\ge \sigma-\del, \im(s)\ge 0\}$,  then
$$\lim_{X\to\infty}\frac{\m(P_1(X))}{X}=0.$$
(ii) If $\xi=\pi$, $\mathcal{B}$ satisfies the conditions of Section \ref{sec:construction},
 $\zeta(\rho)=0$ for all $\rho\in \mathcal{B}$,
and $\zeta(s)$ has no other zeros in the region $\{ s : \re(s)\ge \sigma-\del, \im(s)\ge 0\}$,  then
$$\lim_{X\to\infty}\frac{\m(P_1(X))}{X}=1.$$
\end{thm}

We omit the proof of Theorem \ref{Thm2}, as it is nearly identical to the proof of Theorem
\ref{MainThm} in the case $q=4$.

%
\section{The construction of $\mathcal{B}$}\label{sec:construction}
%

For $j\ge 1$, suppose that
\begin{equation}\label{parameters}
\begin{aligned}
&\exp \(j^8\) \le \gamma_j \le 2 \exp\left(j^{8}\right), \quad
\left|\delta_j-\frac{1}{j^{8}}\right| \le \frac{1}{j^9},\\
 &\text{ and } \quad \left|\theta_j-\frac{\xi-\pi/2}{j^{16}}\right| \le \frac{1}{j^{17}}.\\
\end{aligned}
\end{equation}
We choose $j_0$ so large that for all $j\ge j_0$, $\g_j>A$ and $\sigma-\delta \le \sigma-\delta_j$.
Then we take $\mathcal{B}$ to be the union, over $j\geq j_0$ and $1\leq k\leq j^3$, of
$m(k,j)=k(j^3+1-k)$ copies of $\rho_{j,k}$, where
$$ \rho_{j,k}=\sigma-\delta_j+i(k\gamma_j+\theta_j).$$

%
\section{Preliminary Results}
%

The following classical-type explicit formula was established in
Lemma 1.1 of \cite{FK} when $x'=x$.  The slightly more general result below,
which is more convenient for us, is proved in exactly the same way.

\begin{lem}\label{ExplicitFormula} Let $\beta\geq 1/2$ and for each
 non-principal character $\chi \bmod q$, let $B(\chi)$ be the sequence of
zeros (duplicates allowed) of $L(s,\chi)$ with $\re(s)>\beta$ and $\im(s)>0$.
 Suppose further that all $L(s,\chi)$ are zero-free on the real segment $\b<s<1$.
 If $(a,q)=(b,q)=1$, $x$ is sufficiently large and $x'\ge x$, then
$$
\phi(q)\big(\pi(x;q,a)-\pi(x;q,b)\big)= -
2\re\left(\sum_{\substack{\chi\neq \chi_0 \\ \chi \bmod q}}
(\overline{\chi}(a)-\overline{\chi}(b))
\sum_{\substack{\rho\in B(\chi)\\|\im(\rho)|\leq x'}}f(\rho)\right)+ O\left(x^{\beta}\log^2 x\right),
$$
where
$$
f(\rho):=\frac{x^{\rho}}{\rho\log x}+\frac{1}{\rho}\int_2^x \frac{t^{\rho}}{t\log^2 t}dt=
\frac{x^{\rho}}{\rho\log x}+ O\left(\frac{x^{\re(\rho)}}{|\rho|^2\log^2 x}\right).
$$
\end{lem}

{\bf Remark.}  For Theorem \ref{Thm2}, we use a similar explicit formula for $\pi(x)$ in terms of
the zeros $B(\zeta)$ of the Riemann zeta function which satisfy $\Re \rho > \beta$ and $\Im \rho>0$:
\[
 \pi(x) = \li(x) - 2 \Re \ssum{\rho \in \mathcal{B}(\zeta) \\ |\Im \rho| \le x'} f(\rho) +
O(x^\beta \log^2 x).
\]

Using properties of the Fej\'er kernel we prove the following key proposition.

\begin{pro}\label{Fejer}
Let $\gamma\geq 1$,  $L\ge 4$ and $X\ge 2$.  Define
$$F_{\gamma,L}(x)=\sum_{k=1}^{L-1}(L-k)\cos\left(k\gamma\log x\right).$$
Then
$$\m\left\{x\in [1,X]: F_{\gamma,L}(x)\geq -\frac{L}{4}\right\}\ll \frac{X}{\sqrt{L}}.$$
\end{pro}

\begin{proof}
The Fej\'er kernel satisfies the following identity
$$ \frac1L\left(\frac{\sin\left(\frac{L\theta}{2}\right)}{\sin\left(\frac{\theta}{2}\right)}\right)^2= 1+2\sum_{k=1}^{L-1}\left(1-\frac{k}{L}\right)\cos(k\theta).$$
This yields
$$ F_{\gamma,L}(x)=\frac{\sin^2\left(\frac{L\gamma\log x}{2}\right)}
{2\sin^2\left(\frac{\gamma\log x}{2}\right)}-\frac{L}{2}.$$
Therefore, if $F_{\gamma,L}(x)\ge -L/4$ then
$$ \sin^2\left(\frac{\gamma\log x}{2}\right) \le
\frac{2}{L}\sin^2\left(\frac{L\gamma\log x}{2}\right)\le \frac{2}{L}.$$
Hence,
$$ \left \| \frac{\gamma \log x}{2\pi}\right\|\le \eps :=
 \frac{1}{\sqrt{2L}},$$
where $\|t\|$ denotes the distance to the nearest integer. This implies
\bal
\m\left\{x\in [1,X]: F_{\gamma,L}(x)\geq -\frac{L}{4}\right\}
&\le \m\left\{x\in [1,X]: \quad \left\| \frac{\gamma \log x}{2\pi} \right\| \le \eps\right\} \\
&\le \sum_{0\le k\le \frac{\g \log X}{2\pi}+\eps} e^{2\pi(k+\eps)/\g}-e^{2\pi(k-\eps)/\g}\\
&\ll \frac{\eps}{\g} \sum_{0\le k\le \frac{\g \log X}{2\pi}+\eps} e^{2\pi(k+\eps)/\g}
 \ll \eps X.
\hfill\qedhere
\eal
\end{proof}

%
\section{Proof of Theorem \ref{MainThm}}
%

Suppose $X$ is large and $\sqrt{X}\leq x\leq X$.  For brevity, let
\[
 \Delta = \phi(q)\big(\pi(x;q,a)-\pi(x;q,b)\big).
\]
 It follows from Lemma \ref{ExplicitFormula} with $x'=\max(x,\max \{ j^3\g_j : \g_j\le x\})$
 that
\begin{equation}\label{difference}
\begin{aligned}
\Delta= & -\frac{2}{\log x}\re\left((\overline{\chi}(a)-\overline{\chi}(b))
\sum_{\gamma_j\leq x}\sum_{k=1}^{j^3}
  \frac{x^{\sigma-\delta_j+i(k\gamma_j+\theta_j)}m(k,j)}{\sigma-\delta_j+i(k\gamma_j+\theta_j)}\right)\\
&\qquad + O\(\frac{x^{\sigma}}{\log^2x}\sum_{\gamma_j\leq x}\frac{x^{-\delta_j}}{\g_j^2}
  \sum_{k=1}^{j^3} \frac{m(k,j)}{k^2}   +x^{\sigma-\delta}\log^2 x\)\\
&= \frac{2x^{\sigma}}{\log x}\re\left(i(\overline{\chi}(a)-\overline{\chi}(b))\sum_{\gamma_j\leq x}\frac{x^{-\delta_j}}{\gamma_j}\sum_{k=1}^{j^3}
  x^{i(k\gamma_j+\theta_j)}(j^3+1-k)\right)\\
&\qquad + O\left(\frac{x^{\sigma}}{\log x}\sum_{\gamma_j\leq x}\frac{j^{4}x^{-\delta_j}}{\gamma_j^2}+x^{\sigma-\delta}\log^2 x\right).\\
\end{aligned}
\end{equation}
Note that
$$ \frac{x^{-\delta_j}}{\gamma_j}=\exp\left(-\frac{\log x}{j^{8}}\left(1+O\left(\frac{1}{j}\right)\right)-j^8+O(1)\right).$$
The maximum of this function over $j$ occurs around $J=J(x):=\left[(\log x)^{1/16}\right].$
In this case we have $\log x= J^{16}(1+O(1/J))$ so that
\begin{equation}\label{maximum}
\frac{x^{-\delta_J}}{\gamma_J}= \exp\left(-2J^8+O\left(J^7\right)\right)=\exp\left(-2(\log x)^{1/2}
  + O((\log x)^{7/16})\right).
\end{equation}
We will prove that most of the contribution to the main term on the right hand side of \eqref{difference} comes for the $j$'s in the range $J-J^{3/4}\leq j\leq J+J^{3/4}.$
First, if $j\geq 3J/2$ or $j\leq J/2$ then
$$ \frac{x^{-\delta_j}}{\gamma_j}\ll \exp\left(-4J^8\right)\ll \exp\left(-(\log x)^{1/2}\right)\frac{x^{-\delta_J}}{\gamma_J}.$$
Now suppose that $J/2< j< J-J^{3/4}$ or $J+J^{3/4}<j<3J/2$. Write $j=J+r$ with $J^{3/4}<|r|<J/2$.
For $x>0$, $x+1/x = 2+ (x-1)^2/x$, hence
\[
 \(1+\frac{r}{J}\)^8+\(1+\frac{r}{J}\)^{-8} \ge  \(1+\left|\frac{r}{J}\right|\)^8+\(1+\left|\frac{r}{J}\right|\)^{-8}
\ge 2 + \frac{(8r/J)^2}{1+8r/J} \ge 2 + 12 (r/J)^2.
\]
We infer from \eqref{maximum} that
\begin{equation*}
\begin{aligned}
\frac{x^{-\delta_j}}{\gamma_j}&= \exp\left(-\frac{J^{16}}{j^{8}}\left(1+O\left(\frac{1}{J}\right)\right)-j^8\right)\\
&=\exp\left(-J^{8}\left(\left(1+\frac{r}{J}\right)^8+\left(1+\frac{r}{J}\right)^{-8}\right)+O(J^7)\right)\\
&\leq \exp\left(-2J^{8}\left(1+\frac{6}{\sqrt{J}}\right)+O(J^7)\right)\\
&\ll \exp\left(-2(\log x)^{1/3}\right)\frac{x^{-\delta_J}}{\gamma_J}.\\
\end{aligned}
\end{equation*}
Since $\gamma_j\leq x$ implies that $j\ll (\log x)^{1/8}$, the contribution of the terms $1\leq j< J-J^{3/4}$ or $J+J^{3/4}<j$ to the main term of \eqref{difference} is
\begin{equation}\label{deviation}
\ll \exp\left(-2(\log x)^{1/3}\right)\frac{x^{\sigma-\delta_J}}{\gamma_J}\sum_{j\leq (\log x)^{1/4}}\sum_{k=1}^{j^3}(j^3+1-k)\ll \exp\left(-(\log x)^{1/3}\right)\frac{x^{\sigma-\delta_J}}{\gamma_J}.
\end{equation}
Similarly, we have
\begin{align*}
 \frac{x^{-\delta_j}}{\gamma_j^2}&=\exp\left(-\frac{\log x}{j^{8}}\left(1+O\left(\frac{1}{j}\right)\right)-2j^8+O(1)\right)\\
 &\ll \exp\left(-2\sqrt{2}(\log x)^{1/2}(1+o(1))\right)\\
 &\ll \exp\left(-2(\log x)^{1/3}\right)\frac{x^{-\delta_J}}{\gamma_J},\\
\end{align*}
which follows from \eqref{maximum} along with the fact that the maximum of
  $f(t)=-\log x/t^{8}-2t^8$ occurs at $t=(\log x/2)^{1/16}$. Hence, using \eqref{maximum},
 the contribution of the error term of \eqref{difference} is
\begin{equation}\label{error}
\begin{aligned}
\ll \exp\left(-2(\log x)^{1/3}\right)\frac{x^{\sigma-\delta_J}}{\gamma_J}\sum_{j\leq (\log x)^{1/4}}j^{4}+x^{\sigma-\delta}\log^2 x\ll \exp\left(-(\log x)^{1/3}\right)\frac{x^{\sigma-\delta_J}}{\gamma_J}.
\end{aligned}
\end{equation}

Therefore, inserting the bounds \eqref{deviation} and \eqref{error} in \eqref{difference} we deduce that
\begin{equation}\label{difference2}
\begin{aligned}
\Delta&=\frac{2x^{\sigma}}{\log x}\re\left(i(\overline{\chi}(a)-\overline{\chi}(b))\sum_{|j-J|\le J^{3/4}}\frac{x^{-\delta_j}}{\gamma_j}\sum_{k=1}^{j^3}
\exp\left(i(k\gamma_j+\theta_j)\log x\right)(j^3+1-k)\right)\\
&\qquad + O\left(\exp\left(-(\log x)^{1/3}\right)\frac{x^{\sigma-\delta_J}}{\gamma_J}\right).\\
\end{aligned}
\end{equation}

Let $J-J^{3/4}\leq j\leq J+J^{3/4}$. Then $j^{16}= J^{16}\left(1+O(J^{-1/4})\right).$ Hence we get
\begin{align*}
\theta_j\log x&= \(\arg(\chi(a)-\chi(b))-\frac{\pi}{2}\)\frac{\log x}{j^{16}} + O\pfrac{\log x}{j^{17}} \\
&= \(\arg(\chi(a)-\chi(b))-\frac{\pi}{2}\) + O\pfrac{1}{J^{1/4}}.
\end{align*}
This implies
$$i(\overline{\chi}(a)-\overline{\chi}(b))\exp\left(i\theta_j\log x\right)= |\chi(a)-\chi(b)|\left(1+ O\left(\frac{1}{J^{1/4}}\right)\right),$$
since $e^{i\arg z}=z/|z|.$ Inserting this estimate in \eqref{difference2} we obtain
\begin{equation}\label{difference3}
\begin{aligned}
\Delta=& \left(1+O\left(\frac{1}{\log^{1/64} x}\right)\right)2|\chi(a)-\chi(b)|
\sum_{|j-J|\le J^{3/4}}\frac{x^{\sigma-\delta_j}}{\gamma_j\log x} F_{\g_j,j^3}(x)\\
&+ O\left(\exp\left(-(\log x)^{1/3}\right)\frac{x^{\sigma-\delta_J}}{\gamma_J}\right).\\
\end{aligned}
\end{equation}
For $x\in [\sqrt{X}, X]$ we have $\frac14(\log X)^{1/16}\leq J-J^{3/4}$ and $J+J^{3/4} \leq 4(\log X)^{1/16}$ if $X$ is sufficiently large, since $J=(\log x)^{1/16}+O(1).$
We define
$$\Omega:= \left\{x\in [\sqrt{X},X]: F_{\g_j,j^3}(x)\leq -\frac{j^3}{4} \text{ for all }  \frac14(\log X)^{1/16}\leq j\leq 4(\log X)^{1/16}\right\}.$$
Then it follows from Proposition \ref{Fejer} that
\begin{equation}\label{measure}
\begin{aligned}
 \m \Omega &= X+O\left(X \sum_{\frac14 (\log X)^{1/16}\leq j\leq 4(\log X)^{1/16}}\frac{1}{j^{3/2}}+\sqrt{X}\right)\\
 &= X\left(1+O\((\log X)^{-1/32}\) \).\\
 \end{aligned}
\end{equation}
Furthermore, if $x\in \Omega$ then we infer from \eqref{difference3} that
\begin{align*}
\Delta &\le -\frac{1}{3}|\chi(a)-\chi(b)|\sum_{|j-J|\le J^{3/4}}\frac{j^3x^{\sigma-\delta_j}}{\gamma_j\log x}
+ O\left(\exp\left(-(\log x)^{1/3}\right)\frac{x^{\sigma-\delta_J}}{\gamma_J}\right).\\
&\le -\frac{1}{3}|\chi(a)-\chi(b)|\frac{J^3x^{\sigma-\delta_J}}{\gamma_J\log x}\left(1+ o(1)\right)<0
\end{align*}
if $X$ is sufficiently large, which completes the proof.

%
\section{Acknowledgement}
%

The research of K. F. was partially supported by
National Science Foundation grant DMS-0901339.
The research of S.~K. was partially supported
by Russian Fund for Basic Research, Grant N.~11-01-00329.
The research of Y. L. was supported by a Postdoctoral Fellowship
 from the Natural Sciences and Engineering Research Council of Canada.

\end{document}